\documentclass{amsart}
\usepackage{amsmath}%用于Align环境
\usepackage{amsfonts}%用于\mathbb{}
\usepackage{amsthm}%用于定理环境
\usepackage{amssymb}%第一次用于写右箭头\nrightarrow 命令
\usepackage{enumerate}%计时器，用于编号公式之类;还用于列表环境。
\usepackage{cite}
\theoremstyle{plain}
\newtheorem{thm}{Theorem}[section]%写定理
\newtheorem{prop}[thm]{Property}
\newtheorem{proposition}[thm]{Proposition}
\newtheorem{exam}[thm]{Example}
\newtheorem{cor}[thm]{Corollary}
\newtheorem{lemma}[thm]{Lemma}
\theoremstyle{definition}%定义格式
\newtheorem{defn}[thm]{Definition}
\theoremstyle{remark}%注释格式
\newtheorem{rmk}[thm]{Remark}

\makeatletter%以下四行是用于写大小写罗马数字的，本文暂时只用得到大写。

\newcommand{\Rmnum}[1]{\expandafter\@slowromancap\romannumeral #1@}
\makeatother

\thanks{The first author is supported by NSF of China No.11925107, No.11671057 and No.11688101.}
\begin{document}

\title{Localization of the Kobayashi metric and applications}
\author{Jinsong Liu \textsuperscript{1,2} \:$\&$ Hongyu Wang \textsuperscript{1,2}}
%\thanks{Jinsong Liu }

\address{$1.$ HLM, Academy of Mathematics and Systems Science,
Chinese Academy of Sciences, Beijing, 100190, China}
\address{ $2.$ School of
Mathematical Sciences, University of Chinese Academy of Sciences,
Beijing, 100049, China} \email{liujsong@math.ac.cn, \: wanghongyu16@mails.ucas.ac.cn}

\begin{abstract} {In this paper we introduce a new class of domains --- log-type convex domains, which have no boundary regularity assumptions. Then we will localize the Kobayashi metric in log-type convex subdomains. As an application, we prove a local version of continuous extension of rough isometric maps between two bounded domains with log-type convex Dini-smooth boundary points. Moreover we prove that the Teichm\"uller space $\mathcal T_{g,n}$ is not biholomorphic to any bounded pseudoconvex domain in $\mathbb C^{3g-3+n}$ which is locally log-type convex near some boundary point.}
\end{abstract}

\maketitle

\section{\noindent{{\bf Introduction}}}\label{In}
\noindent In the complex plane $\mathbb{C}$, if domains $\Omega_1$ and $\Omega_2$ are bounded by closed Jordan curves, then every biholomorphic map $f:\Omega_1\rightarrow \Omega_2$ extends to a homeomorphism of $\overline{D}_1$ onto $\overline{D}_2$. In $\mathbb{C}^n (n\geq 2)$ the problem is more interesting and difficult. If $\Omega_1, \: \Omega_2$ are bounded pseudoconvex domains with
$$\frac{1}{C}\delta_{\Omega_1}^{\frac{1}{\nu_1}}(z)\leq\delta_{\Omega_2}(f(z))\leq C\delta_{\Omega_1}^{\nu_1}(z), \:\:\: \forall z\in\Omega_1,$$ and the {\it Kobayashi metric} satisfies
$$k_{\Omega_2}(\omega,v)\geq\frac{C|v|}{{\delta_{\Omega_2}(\omega)}^{\nu_2}}, \:\:\:\: \forall \omega\in\Omega_2, \:\: v\in\mathbb{C}^n, $$
for some $\nu_1, \:\nu_2, \:C>0$, where
$\delta_{\Omega_i}(z):=\inf\{|w-z|, w \in \partial \Omega_i\}, \:
i=1,2$, then the proper holomorphic map $f:\Omega_1\rightarrow
\Omega_2$ extends to a H\"{o}lder continuous map of
$\overline{\Omega}_{1}$. This result holds in particular if $\Omega_i$ are
{\it strictly pseudoconvex domains} and more generally pseudoconvex
domains with {\it finite type}. There are many other generalizations, and we
refer the reader to the survey
\cite{forstneric1993proper} by F. Forstneri\u{c}.

In \cite{Forstneric1987Localization}, F. Forstneri\u{c} and J. P.
Rosay firstly proved a local version of the continuous extension. If
$f:\Omega_1\rightarrow \Omega_2$ is a proper holomorphic map,
$\xi\in\partial \Omega_1$ and $w_0\in\partial \Omega_2$ are $C^2$
strictly pseudoconvex points and $w_0\in C(f,\xi)$ (i.e. the
cluster set of all limit points of $f$ as points in $\Omega_1$
approach to $\xi \in \partial \Omega_1$), then $f$ extends to $\xi$
continuously. After that, A. B. Sukhov generalized the local result to
the case of finite type in \cite{sukhov1995boundary}.

One of the goals of this article is to provide a local version of the latter extension results
along the lines of the work described in the previous paragraph. For this, we need a
localization theorem for the Kobayashi metric in the style of Forstneri\u{c}-Rosay \cite{Forstneric1987Localization} (also see
\cite{Zimmer2016Gromov,Zimmer2017The} by Zimmer). It turns out that there is a natural class of bounded domains, which admits
(in contrast to the domains in the latter results) domains with non-smooth boundary, as well
as domains with boundary points of infinite type, for which such a localization theorem can
be given. It is for these considerations that we focus on domains that are locally {\it log-type}
convex.

\bigskip
At first we give the definition of log-type convex domains.
\begin{defn}\label{logconvex}
A bounded convex domain $\Omega \subset \mathbb C^{n}, \:n\geq2$, is called log-type convex if
\begin{align}
\delta_{\Omega}(z;v)\leq\frac{C}{|\log\delta_{\Omega}(z)|^{1+\nu}}, \:\:\: z\in\Omega,\: v\in \mathbb{C}^n
\end{align}
for some $\nu, \:C>0$. See Section 2 for the definitions of $\delta_{\Omega}(z)$ and $\delta_{\Omega}(z;v)$.

\end{defn}
\begin{rmk}
In \cite{MercerComplex}, P. R. Mercer firstly introduced '$m$-convex' domains. A bounded convex domain is called $m$-convex if there exists $C>0$ such that $\delta_{\Omega}(z;v)\leq C\delta_{\Omega}^{\frac{1}{m}}(z)$ for any $z\in\Omega, v\in\mathbb{C}^n$.
\end{rmk}
\begin{exam}
$\Omega=\{(z_1,z_2)\in\mathbb{C}^2:Rez_1>e^{-\frac{1}{|z_2|^2}}\}$, then $\Omega\cap B_{1}(0)$ is log-type convex.
\end{exam}
Note that in this definition there is no boundary regularity assumptions. In \cite{bharali2016complex}
G. Bharali proved the (global) continuous extension problem of the log-type convex domain with $C^1$ smooth boundary. Then in \cite{Bharali2017Goldilocks}, G. Bharali and A. Zimmer introduced
a class of pseudoconvex domains named {\it 'Goldilocks' domains}, and
proved a lot of properties of such domains. In particular, each global
log-type convex domain is a 'Goldilocks' domain.

However, in this paper we are interested in results that follow from the local properties
of the boundary of the domain considered. Now, a localized version of the 'Goldilocks'
property is hard to formulate in a way that is useful. But we noticed that the property
that Bharali-Zimmer have used often in the proofs in \cite{Bharali2017Goldilocks}, viz. a form of visibility, \emph{does}
admit a localized form (in this regard, see \cite{BharaliA} by Bharali-Maitra whose proofs rely purely on the latter visibility property). It turns out that domains that are locally log-type convex exhibit
the localized property of visibility $-$ which will be very useful. Thus, we introduce the
\emph{locally log-type convex} domains. If $\Omega \subset \mathbb C^n$ is a bounded domain, $\Omega$ is locally log-type
convex if $\Omega$ is log-type convex around some boundary point: i.e. there exists a connected
open set $U$ with $\Omega\cap \partial U\neq\emptyset$  such that $\Omega\cap U$ is log-type convex.

We will follow the method of F. Forstneri\u{c}-J. P. Rosay
\cite{Forstneric1987Localization} and A. Zimmer
\cite{Zimmer2016Gromov} to generalize the local version of continous
extensions. Firstly we prove a localization result near those log-type convex points.
\begin{thm}\label{localdistance}
Let $\Omega$ be a bounded domain in $\mathbb{C}^n$. Suppose that there exists a connected
open set $U$ with $\Omega\cap \partial U\neq\emptyset$ and $\Omega\cap U$ is log-type convex. For any open set $V\subset U$ with $\overline V \subset U$, there exists $K>0$ such that the Kobayashi distances satisfies
$$K_{\Omega}(p,q)\leq K_{\Omega\cap U}(p,q)\leq K_{\Omega}(p,q)+K, $$
for any $p, \:q\in V\cap \Omega$.
\end{thm}
\noindent
We can define,
in analogy with Definition \ref{logconvex}, log-type $\mathbb{C}$-convex domains. We refer the reader to Section 3
for a formal definition, and present the following corollary to Theorem \ref{localdistance}.
\begin{cor}\label{locallocal}
Suppose that $\Omega_1, \: \Omega_2$ are bounded domains in
$\mathbb{C}^n$ and $f$ is a roughly isometric embedding from
$\Omega_1$ to $\Omega_2$. Furthermore, suppose that $\Omega_1\cap
U_1$ is a log-type convex domain with Dini-smooth boundary, and
$\Omega_2\cap U_2$ is a log-type $\mathbb{C}$-convex domain with
Dini-smooth boundary, where $U_i \subset \mathbb C^n, \:\: U_i \cap \partial\Omega_i\neq \emptyset, \:\: i=1,2$.

If $\xi\in\partial\Omega_1\cap U_1$, $\zeta\in\partial\Omega_2\cap
U_2$ and $\zeta\in\mathcal{C}(f,\xi)$, then $f$ extends continuously
to $\xi$.
\end{cor}

\bigskip
Another application is on the local boundary property of domains biholomorphic to the {\it Teichm\"{u}ller space}.

Let $S$ be a surface of finite type $(g, n)$, i.e., an oriented finite genus surface with $n$ punctures. The Teichm\"uller space $\mathcal T_{g, n}$ is the set of marked complex structures on $S$. L. Ahlfors, D. Spencer, K. Kodaira. L. Bers proved that there is a natural complex structure on $\mathcal T_{g, n}$. Furthermore, Bers proved that $\mathcal T_{g, n}$ is biholomorphic to a bounded pseudoconvex domain in $\mathbb C^{3g-3+n}$. However, it is not explicit and not known how smooth the boundary of $\mathcal T_{g, n}$ is.

Recently, V. Markovic proved in \cite{markovic2018caratheodory} that the Kobayashi metric and the Caratheodory metric is not identical on $\mathcal T_{g,n}$. Combining with a result of L. Lempert \cite{lempert1981metrique}, he proved that the Teichm\"uller space $\mathcal T_{g,n}$ is not biholomorphic to a bounded convex domain in $\mathbb C^{3g-3+n}$.
Moreover, by using the deep result of the ergodicity of the Teichm\"{u}ller geodesic flow, S. Gupta and H. Seshadri proved in \cite{Gupta2017On} that $\mathcal T_{g,n}$ cannot be biholomorphic to a bounded domain $\Omega \subset \mathbb C^{3g-3+n}$ which is locally strictly convex at some boundary point.

We prove another result about Teichm\"{u}ller space without using the ergodicity of the Teichm\"{u}ller geodesic flow. The technique of the proof mainly comes from A. Zimmer \cite{Zimmer2018Smoothly}.
\begin{thm}\label{Teichmuller}
The Teichm\"uller space $\mathcal T_{g,n}$ cannot be biholomorphic to a bounded pseudo-convex domain $\Omega \subset \mathbb C^{3g-3+n}$ which is locally log-type convex at some boundary point.
\end{thm}

%%%%%%%%%%%%%%%%%%%%%%%%%%%%%%%%%%%%%%%%%%%%%%%%%%%%%%%%%%%%%%
%\begin{defn}
%A bounded domain $\Omega \subset \mathbb{C}^{n}$ is a Goldilocks domain if
%
%(1)for some (hence any) $\epsilon>0$ we have
%$$\int_{0}^{\epsilon} \frac{1}{r} M_{\Omega}(r) d r<\infty$$
%(2)for each $x_{0} \in \Omega$ there exist constants $C, \alpha>0$ (that depend on $x_{0} )$ such that
%$$K_{\Omega}\left(x_{0}, x\right) \leq C+\alpha \log \frac{1}{\delta_{\Omega}(x)}$$
%for all $x \in \Omega$.
%where $$M_{\Omega}(r) :=\sup \left\{\frac{1}{k_{\Omega}(x ; v)} : \delta_{\Omega}(x) \leq r,\|v\|=1\right\}$$
%\end{defn}
%In this paper we are interested in a class of domains which locally have the properties of Goldilocks convex domains. we localize the Kobayashi distance and then as a corallary get a continuous extension result of the isometric embedding near the log-type subdomain.
\section{{\bf Preliminaries}}\label{BM}
\subsection{Notations}\

(1) \:For $z\in\mathbb{C}^n$ let $|z|$ be the standard Euclidean
norm and $d_{euc}(z_1,z_2)= |z_1-z_2 |$ be the standard Euclidean
distance.

(2) \:Given an open set $\Omega\subset\mathbb{C}^n,\:x\in\Omega$ and
$v\in\mathbb{C}^n\backslash\{0\}$, let
$$\delta_{\Omega}(x)=\inf\{d_{euc}(x,\xi):\xi\in\partial \Omega\}$$
as before, and let
$$\delta_{\Omega}(x,v)=\inf\{d_{euc}(x,\xi):\xi\in\partial \Omega\cap(x+\mathbb{C}v)\}.$$

(3) \:For any curve $\sigma$ we denote by $L(\sigma)$ the
length of $\sigma$.

(4) \:For any $z_0 \in \mathbb C^n$ and $\epsilon >0$, we denote by $B_{\epsilon}(z_0)$ the open ball
$B_{\epsilon}(z_0)=\{z\in \mathbb C^n| \:|z-z_0|<\epsilon\}$.
\subsection{The Kobayashi metric}\

Given a domain $\Omega \subset \mathbb{C}^{n}$, the (infinitesimal)
Kobayashi metric is the pseudo-Finsler metric defined by
$$k_{\Omega}(x ; v)=\inf \left\{|\xi| : f \in \operatorname{Hol}(\mathbb{D}, \Omega), \: f(0)=x,
d(f)_{0}(\xi)=v\right\}.$$ Define the length of any curve $\sigma:[a,b]\rightarrow \Omega$
to be
$$L(\sigma)=\int_{a}^{b} k_{\Omega}\left(\sigma(t) ; \sigma^{\prime}(t)\right) d
t.$$
S. Venturini \cite{VenturiniPseudodistances} proved that
the Kobayashi pseudo-distance can be given by
\begin{align*}
K_{\Omega}(x, y)&=\inf_\sigma \big\{L(\sigma)| \:\sigma :[a, b]
\rightarrow \Omega \text { is any absolutely continuous curve }\\
& \text { with } \sigma(a)=x \text { and } \sigma(b)=y \big\}.
\end{align*}
%\begin{align*}
%K_{\Omega}(x, y)&=\inf \left\{L(\sigma)| \:\sigma :[a, b]
%\rightarrow \Omega \text { is any absolutely continuous curve }.\right \\
%& \text { with } \sigma(a)=x \text { and } \sigma(b)=y \}.
%\end{align*}
Its proof is based on an observation due to H. L. Royden \cite{royden1971remarks}.

There are well known estimates on the Kobayashi metrics on convex domains.
\begin{proposition}
Suppose that $\Omega \subset \mathbb{C}^{n}$ is an bounded convex domain. Then, for any $x, \:y \in \Omega, \:v \in \mathbb C^n$,
\begin{align}
\frac{|v|}{2\delta_{\Omega}(x ; v)} \leq k_{\Omega}(x ; v) \leq \frac{|v|}{\delta_{\Omega}(x ; v)},\\
K_{\Omega}(x, y)\geq \frac{1}{2}\log \left(\frac{|x-\xi|}{|y-\xi|}\right),
\end{align}
where $\xi \in \partial\Omega \cap \{x+(x-y)\cdot\mathbb C\}$.
\end{proposition}
\begin{proposition}[Proposition 2.3, \cite{MercerComplex}]\label{mercer}
Suppose that $\Omega \subset \mathbb{C}^{n}$ is a bounded convex domain and fix $z_0\in\Omega$. Then there exists $\alpha,\beta>0$ such that: $\forall z\in\Omega$,
\begin{align}K_{\Omega}(z,z_0)\leq \alpha+\beta \log\frac{1}{\delta_{\Omega}(z)}.
\end{align}
\end{proposition}
\subsection{Almost geodesics}
\begin{defn}
Suppose $\Omega$ is a bounded domain. If $I\subset\mathbb{R}$ is an
interval, a map $\sigma: I\rightarrow \Omega$ is called a {\it geodesic}
if, for all $s,t\in I$,
$$K_{\Omega}(\sigma(s),\sigma(t))=|t-s|.$$
And $\sigma$ is called a {\it rough geodesic} if there exists $C>0$ such
that
$$|t-s|-C\leq K_{\Omega}(\sigma(s),\sigma(t))\leq|t-s|+C.$$
For $\lambda \geq 1$ and $\kappa \geq 0$, a curve $\sigma : I
\rightarrow \Omega$ is called an $(\lambda, \kappa)$
-{\it almost-geodesic} if:

(1) for all $s, t \in I $
 $$\frac{1}{\lambda}|t-s|-\kappa \leq K_{\Omega}(\sigma(s), \sigma(t)) \leq
 \lambda|t-s|+\kappa;$$

(2) $\sigma$ is absolutely continuous (whence $\sigma^{\prime}(t)$ exists for almost every $t \in I ),$ and for almost every $t \in I$,
$$k_{\Omega}\left(\sigma(t) ; \sigma^{\prime}(t)\right) \leq
\lambda.$$
\end{defn}
In order to give a local estimate of the Kobayashi distance, we need
the properties of geodesics. However, for a general bounded domain
$(\Omega, K_{\Omega})$, it may not be Cauchy complete. Furthermore, it's not
clear whether there is a geodesic between any two points.
Fortunately, G. Bharali and A. Zimmer \cite{Bharali2017Goldilocks} proved
 that there is a $(1,\kappa)$-almost
geodesic between any two points in a bounded domain.
\begin{lemma}\label{almost}
Suppose that $\Omega \subset \mathbb{C}^{n}$ is a bounded domain.
For any $\kappa>0$ and $x, y \in \Omega$, there exists a $(1,
\kappa)$-almost geodesic $\sigma :[a, b] \rightarrow \Omega$ with
$\sigma(a)=x$ and $\sigma(b)=y$.
\end{lemma}
\subsection{Gromov Product}
\begin{defn}
Let $(X, d)$ be a metric space. Given three points $x, y, o \in$ $X,$ the {\it Gromov product} is given by
$$(x | y)_{o}=\frac{1}{2}\Big(d(x, o)+d(o, y)-d(x, y)\Big).$$
A proper geodesic metric space $(X, d)$ is {\it Gromov hyperbolic} if and
only if there exists $\delta \geq 0$ such that, for all $o, x, y, z
\in X$,
$$(x | y)_{o} \geq \min \left\{(x | z)_{o},(z | y)_{o}\right\}-\delta.$$

\end{defn}
\begin{rmk}
In \cite{balogh2000gromov}, Z. Balogh and M. Bonk proved those strongly pseudoconvex domains with the Kobayashi metric are Gromov hyperbolic. Later A. Zimmer \cite{Zimmer2016Gromov} proved that smooth convex domains with the Kobayashi metrics are Gromov hyperbolic if and only if they are of finite type. Furthermore, he got some other results about Gromov product in \cite{Zimmer2017The}.
\end{rmk}
\begin{thm}[Thereom 4.1 in \cite{Zimmer2017The}]
Let $\Omega \subset \mathbb{C}^{n}$ be a bounded convex domain
with $C^{1, \epsilon}$ boundary and $o \in \Omega$. Suppose $\{z_{k}\}, \{w_{m}\}$ are
sequences in $\Omega$ such that $z_{k} \rightarrow x \in \partial
\Omega$ and $w_{m} \rightarrow y \in \partial \Omega.$ Then:

(1) If $x=y,$ then the Gromov product
$$\lim _{k, m \rightarrow \infty}\left(z_{k} | w_{m}\right)_{o}=\infty;$$

(2) If
$$\limsup _{k, m \rightarrow \infty}\left(z_{k} | w_{m}\right)_{o}=\infty,$$
then $T_{x}^{\mathbb{C}} \partial \Omega=T_{y}^{\mathbb{C}} \partial \Omega$, where  $T_{x}^{\mathbb{C}} \partial \Omega$
is the complex affine hyperplane tangent to $\partial\Omega$ at $x$.
\end{thm}
\subsection{Squeezing Function}
Given a domain $\Omega \subset \mathbb{C}^{n}$, let $s_{\Omega} :
\Omega \rightarrow(0,1]$ be the {\it squeezing function} on $\Omega,$ defined by
\begin{align*} s_{\Omega}(z)=& \text { sup }\{r| \text { there exists a 1-1 holomorphic map }\\ & f : \Omega \rightarrow \mathbb{B}_{d}(0 ; 1) \text { with } f(z)=0 \text { and } \mathbb{B}_{d}(0 ; r) \subset f(\Omega) \}. \end{align*} See \cite{Deng2011On, Yeung2009Geometry} for the details.
\begin{rmk}
From Bers Embedding Theorem, it follows that the squeezing function of the Teichm\"{u}ller space has a uniform positive bound from below.
\end{rmk}
\begin{lemma}[Yeung\cite{Yeung2009Geometry}]\label{squeezing}
Let $\Omega \subset \mathbb{C}^{n}$ be a bounded domain. If the squeezing function
$s_{\Omega}(\xi)>s,$ then the Kobayashi metric, Bergman metric and $K\ddot{a}hler$-Einstein metric are bilipschitz on $\{z \in \Omega : K_{\Omega}\left(z, z_{0}\right) \leq \epsilon\}$, and there exists $c>0, \:\epsilon>0$ such that
$$Vol \left(\left\{z \in \Omega : K_{\Omega}\left(z, z_{0}\right) \leq \epsilon\right\}\right) \geq \frac{\epsilon^{2 n}}{c},$$
where $Vol$ denotes the volume with respect to either the Bergman metric, the $K\ddot{a}hler$-Einstein metric or the Kobayashi-Eisenman metric.
\end{lemma}

\section{Localization of the Kobayashi distance }
\noindent F. Forstneri$\check{c}$ and J. P. Rosay \cite{Forstneric1987Localization} gave a local estimate of the Kobayashi metric near a local peak point under some growth condition for the peak function. Then Zimmer generalized the result to convex domain of finite type in \cite{Zimmer2016Gromov}. By adopting an analogous method as in \cite{Forstneric1987Localization,Zimmer2016Gromov}, we will give a local estimation of the Kobayashi metric in log-type convex domains.
\begin{lemma}\label{Lemma1}
Suppose that $\Omega \subset \mathbb{C}^n$ is a bounded domain and $\Omega\cap U$ is log-type convex. Then for any $\eta>0$, there exists $\alpha>0, \:\tau>0$ such that: if $\xi\in \partial\Omega\cap U$, $B_{\tau}(\xi)\subset U$ and $\varphi:\mathbb{D}\rightarrow B_{\tau}(\xi)\cap\Omega$, and
$$|\zeta|\leq1-\frac{\alpha}{|\log\delta_{\Omega}(\varphi(0))|^{1+\nu}},$$
then we have
$$|\varphi(\zeta)-\varphi(0)|\leq\eta.$$
\end{lemma}
\begin{proof}
Let $\displaystyle{\alpha=\frac{3}{\eta}}$.
From
$$1-\frac{\alpha}{|\log\delta_{\Omega}(\varphi(0))|^{1+\nu}}>0,$$
it follows that
$$|\log\delta_{\Omega}(\varphi(0))|>{\alpha}^{\frac{1}{1+\nu}}.$$
Select $v\in\mathbb{C}^n$ such that $\varphi(\zeta)\in\varphi(0)+\mathbb{C}v$.

Firstly, supposing that $$|\varphi(\zeta)-\varphi(0)|\geq 2\delta_{\Omega\cap U}(\varphi(0);v),$$  we obtain
\begin{equation}\label{A}
K_{\Omega\cap U}(\varphi(0),\varphi(\zeta))\geq\frac{1}{2}\log\frac{\left||\varphi(\zeta)-\varphi(0)|-\delta_{\Omega\cap U}(\varphi(0);v)\right|}{\delta_{\Omega\cap U}(\varphi(0);v)}.
\end{equation}
Moreover, by the decreasing property of the Kobayashi metric, it follows that
\begin{equation}\label{B}
K_{\Omega\cap U}(\varphi(\zeta),\varphi(0))\leq d_{\mathbb{D}}(0,\zeta)=\frac{1}{2}\log\frac{1+|\zeta|}{1-|\zeta|}\leq\frac{1}{2}\log\frac{2}{1-|\zeta|}.
\end{equation}
Application of the previous two inequalities (\ref{A}) and (\ref{B}) now gives
\begin{align*}
|\varphi(\zeta)-\varphi(0)|&\leq\left(\frac{2}{1-|\zeta|}+1\right)\delta_{\Omega\cap U}(\varphi(0);v)
\leq\left(\frac{2}{1-|\zeta|}+1\right)\frac{1}{|\log\delta_{\Omega\cap U}(\varphi(0))|^{1+\nu}}\\
&\leq\left(\frac{2}{1-|\zeta|}+1\right)\frac{1}{|\log\delta_{\Omega}(\varphi(0))|^{1+\nu}}
\leq\frac{2}{\alpha}+\frac{1}{|\log\delta_{\Omega}(\varphi(0))|^{1+\nu}}\\
&\leq\frac{3}{\alpha}=\eta.
\end{align*}

Now, if $$|\varphi(\zeta)-\varphi(0)|< 2\delta_{\Omega\cap U}(\varphi(0);v).$$ then by Definition \ref{logconvex}, there exists $C,\nu>0$ and $$\delta_{\Omega\cap U}(\varphi(0);v)<\frac{C}{|\log\delta_{\Omega\cap U}(\varphi(0))|^{1+\nu}}.$$
We can choose $\tau$ small enough such that $\forall z\in B_{\tau}(\xi)\cap \Omega$,
$$\frac{C}{|\log\delta_{\Omega\cap U}(z)|^{1+\nu}}<\frac{1}{2}\eta.$$
Thus $$|\varphi(\zeta)-\varphi(0)|<\eta,$$
which completes the proof.
\end{proof}

\begin{thm}\label{local 3}
Suppose that $\Omega \subset \mathbb{C}^n$ is a bounded domain. If $\Omega\cap U$ is log-type convex. then there exist $c>0$ and $\epsilon>0$ such that
$$k_{\Omega}(p;v)\leq k_{\Omega\cap U}(p;v)\leq e^{c|\log\delta_{\Omega}(\varphi(0))|^{-(1+\nu)}} k_{\Omega}(p;v)$$
for $\xi\in\partial\Omega\cap U$ and $p\in B_{\epsilon}(\xi)\cap \Omega$.
\end{thm}
\begin{proof}
Assume there exists $\eta>0$ such that for any $p\in B_{\tau}(\xi)$, $B_{\eta}(p)\subset U$.
Define
$$\rho(\epsilon)=max\{r:\varphi\in \text{Hol}(\mathbb{D},\Omega), \varphi(0)\in B_{\epsilon}(\xi), |\varphi(0)-\varphi(\zeta)|\leq\eta \text{ for any } |\zeta|\leq r\}.$$
Then we only need to check that $\rho\geq e^{-c|\log{\epsilon}|^{-(1+\nu)}}$.

Scaling domain as necessary, we assume that $diam_{euc}(\Omega)\leq 1$. Then Schwarz Lemma implies that $\rho\geq \eta$. Now by Lemma \ref{Lemma1}, it follows that there exists $\alpha>0$ such that: if
$$ |\zeta|\leq \rho-\alpha|\log\epsilon|^{-(1+\nu)},$$
then $|\varphi(0)-\varphi(\zeta)|\leq\frac{\eta}{2}$.
If $\rho<1$, then there exists a holomorphic map $\varphi:\mathbb{D}\rightarrow\Omega$ such that $\varphi(0)\in B_{\epsilon}(\xi)$ and
$$\eta=\sup\limits_{|\zeta|=\rho}|\varphi(\zeta)-\varphi(0)|.$$
Hadamard's three circle lemma now gives that
$$M(r)=\log\sup\limits_{|\zeta|=r}|\varphi(\zeta)-\varphi(0)|$$
is a convex function of $\log(r)$. Noting $\rho-\alpha(\log\epsilon)^{-(1+\nu)}<1$, we then obtain
$$\frac{\log(\rho-\alpha|\log\epsilon|^{-(1+\nu)})}{\log\frac{\eta}{2}}\geq \frac{\log(\rho-\alpha|\log\epsilon|^{-(1+\nu)})}{M(\rho-\alpha|\log\epsilon|^{-(1+\nu)})}\geq\frac{\rho}{M(\rho)}=\frac{\rho}{\log\eta}.$$
Therefore
\begin{equation}\label{C}
\log(\rho-\alpha|\log\epsilon|^{-(1+\nu)})\leq\frac{\log\frac{\eta}{2}}{\log\eta}\log\rho.
\end{equation}
Assuming $\epsilon\leq \exp\left(-(\frac{2\alpha}{\eta})^{\frac{1}{1+\nu}}\right)$, we have $\rho-\alpha|\log\epsilon|^{-(1+\nu)}\geq\frac{\eta}{2}$, which implies that
\begin{equation}\label{D}
\log\left(\rho-\alpha|\log\epsilon|^{-(1+\nu)}\right)\geq\log \rho-\frac{2}{\eta}\frac{\alpha}{|\log\epsilon|^{(1+\nu)}}.
\end{equation}
Combination of (\ref{C}) and (\ref{D}) now gives the desired result
$$\rho\geq \exp \left({-c|\log{\epsilon}|^{-(1+\nu)}}\right), $$
where $c=\frac{2\alpha\log\frac{1}{\eta}}{\eta\log2}$.
\end{proof}

\begin{lemma}\label{local}
Suppose that $\Omega \subset \mathbb{C}^n$ is a bounded domain and $\Omega\cap U$ is log-type convex. Then there exists $K>0$ with the following property:

If $\xi\in\partial\Omega\cap U$ and $\sigma:[a,b]\rightarrow\Omega$ is a $(1,\kappa)$-almost geodesic with $\sigma([a,b])\subset B_{\epsilon}(\xi) $ then, for all $s,t\in[a,b]$,
$$|t-s|\leq K_{\Omega\cap U}(\sigma(s),\sigma(t))\leq|t-s|+K.$$

\end{lemma}
\begin{proof}
The left side is obvious. To prove the right side, let $T\in[a,b]$ satisfy
$$\delta_{\Omega}(\sigma(T))=\max\{\delta_{\Omega}(\sigma(t)):t\in[a,b]\}.$$
Fix $z_0\in\Omega$. By Proposition \ref{mercer}, there exists $\alpha,\beta>0$ and if $\sigma|_{[a,b]}$ is near the boundary, we have
\begin{align}
|T-t|&=K_{\Omega}(\sigma(T),\sigma(t))+\kappa\leq K_{\Omega}(\sigma(T),z_0)+K_{\Omega}(z_0,\sigma(t))+\kappa\notag\\
&\leq 2\alpha+\kappa+2\beta \log\frac{1}{(\delta_{\Omega}(\sigma(t))\delta_{\Omega}(\sigma(T)))^{1/2}}\\
&\leq (2\beta +1) \log\frac{1}{(\delta_{\Omega}(\sigma(t))\delta_{\Omega}(\sigma(T)))^{1/2}}-1,\notag
\end{align}
which implies that
$$\delta_{\Omega}(\sigma(t))\leq (\delta_{\Omega}(\sigma(t))\delta_{\Omega}(\sigma(T)))^{1/2}\leq e^{\frac{-1}{2\beta  +1}(|T-t|+1)}.$$
It follows immediately from Theorem \ref{local 3} that
\begin{align*}
K_{\Omega\cap U}(\sigma(s),\sigma(t))&\leq \int_s^t k_{\Omega\cap U}(\sigma(r);\sigma'(r))dr\leq \int_s^t e^{c|\log\delta_{\Omega}(\sigma(r))|^{-(1+\nu)}}k_{\Omega}(\sigma(r);\sigma'(r))dr\\
&\leq \int_s^t e^{c|\log\delta_{\Omega}(\sigma(r))|^{-(1+\nu)}}dr\leq \int_s^t e^{\frac{c(2\beta  +1)^{(1+\nu)}}{(|T-t|+1)^{(1+\nu)}}}dr\\
&\leq \int_{[s,t]\cap[T-1,T+1]} e^{\frac{c(2\beta  +1)^{(1+\nu)}}{(|T-t|+1)^{(1+\nu)}}}dr+\int_{[s,t]\cap[T-1,T+1]^c} e^{\frac{c(2\beta  +1)^{(1+\nu)}}{(|T-t|+1)^{(1+\nu)}}}dr\\
&\leq2 e^{c(2\beta  +1)^{(1+\nu)}}+\int_{[s,t]\cap[T-1,T+1]^c} e^{\frac{c(2\beta  +1)^{(1+\nu)}}{(|T-t|+1)^{(1+\nu)}}}dr.
\end{align*}
Notice that, for $\lambda\in[0,1]$,
$$e^{C\lambda}=1+\int_{0}^{\lambda}Ce^{Cs}ds\leq 1+\int_0^{\lambda}Ce^C ds\leq 1+Ce^C\lambda.$$
Thus, by setting $C=c(2\beta  +1)^2$, we obtain
$$\exp \left({\frac{c(2\beta  +1)^{(1+\nu)}}{(|T-t|+1)^{(1+\nu)}}}\right)\leq1+\frac{Ce^C}{(|T-t|+1)^{(1+\nu)}},$$
which implies that
\begin{align*}
& \int_{[s,t]\cap[T-1,T+1]^c}\exp \left({\frac{c(2\beta  +1)^{(1+\nu)}}{(|T-t|+1)^{(1+\nu)}}}\right)dr\\
\leq& \int_{[s,t]\cap[T-1,T+1]^c}1+\frac{Ce^C}{(|T-t|+1)^{(1+\nu)}}ds\\
\leq& |t-s|+\int_1^{\infty}\frac{Ce^C}{(r+1)^{(1+\nu)}}dr.
\end{align*}
Therefore, we have the desired result
$$K_{\Omega\cap U}(\sigma(s),\sigma(t))\leq|t-s|+K,$$
where 
$$K=2 \exp\left({c(2\beta  +1)^{(1+\nu)}}\right)+\int_1^{\infty}\frac{2c(2\beta  +1)^{(1+\nu)}}{(r+1)^{(1+\nu)}}e^{c(2\beta  +1)^{(1+\nu)}}dr.
$$

The proof is complete.
\end{proof}

Recall that $\Omega$ is a bounded domain and $\Omega\cap U$ is log-type convex. Then $(\Omega,K_{\Omega})$ has the following local visible property. Note that Bharali and Zimmer \cite{Bharali2017Goldilocks} proved the property for all Goldilocks domains. For
the sake of completeness, we present their proof here.
\begin{lemma}\label{visibility}
For any $\xi\neq\xi'\in\partial\Omega\cap U$, there exists $\epsilon>0$ and a compact set $A\subset \Omega\cap U$ with the following property:

For any $z\in\Omega\cap B_{\epsilon}(\xi)$, $\omega\in\Omega\cap B_{\epsilon}(\xi')$ and a $(1,\kappa)$-almost geodesic $\sigma:[0,T]\rightarrow\Omega$ joining $z$ and $\omega$, then $\sigma\cap A\neq\emptyset$.
\end{lemma}
%\begin{rmk}\label{local2}

%\end{rmk}
\begin{proof}
Fix $z_0\in\Omega$, and choose $\epsilon$ small enough such that $B_{\epsilon}(\xi)\cap B_{\epsilon}(\xi')=\emptyset$ and $\forall z\in\Omega\cap B_{\epsilon}(\xi)$,
$$\delta_{\Omega}(z)=\delta_{\Omega\cap U}(z).$$
Taking $z\in\Omega\cap B_{\epsilon}(\xi)$ and $\omega\in\Omega\cap B_{\epsilon}(\xi')$, let $\sigma:[0,T]\rightarrow\Omega$ be a $(1,\kappa)$-almost geodesic joining $z$ and $\omega$. By denoting
$$T_0=\max\left\{t\in[0,T]:\sigma([0,t])\subset \overline{B_{\epsilon}(\xi)}\right\},$$
we choose $\tau\in[0,T_0]$ which satisfies
$$\delta_{\Omega}(\sigma(\tau))=\max\{\delta_{\Omega}(\sigma(t)):t\in[0,T_0]\}.$$
Now for $t\in[0,T_0]$, in view of Proposition \ref{mercer} we have
\begin{align}
|t-\tau|&=K_{\Omega}(\sigma(t),\sigma(\tau))+\kappa\leq K_{\Omega}(\sigma(t),z_0)+K_{\Omega}(z_0,\sigma(\tau))+\kappa\notag\\
&\leq2\alpha+\kappa+\beta\log\frac{1}{\delta_{\Omega}(\sigma(t))\delta_{\Omega}(\sigma(\tau))},
\end{align}
for some $\alpha,\beta>0.$
Therefore
$$\delta_{\Omega}(\sigma(t))\leq \sqrt{\delta_{\Omega}(\sigma(t))\delta_{\Omega}(\sigma(\tau))}\leq \exp\left(\frac{-|t-\tau|+2\alpha+\kappa}{2\beta}\right).$$
Noting that $\Omega\cap U$ is convex, if $z$ is near $\xi$, then we have
 \begin{align}
 k_{\Omega\cap U}(z,v)\geq\frac{|v|}{2\delta_{\Omega\cap U}(z,v)}=\frac{|v|}{2\delta_{\Omega}(z,v)}.
 \end{align}
By Theorem $\ref{local 3}$, it follows that there exists $c_1>0$ such that, for any $z$ near $\xi$,
\begin{align*}
k_{\Omega\cap U}(z,v)\leq c_1k_{\Omega}(z,v).
\end{align*}
Suppose that $$\delta_{\Omega\cap U}(z,v)\leq c_2|\log\delta_{\Omega\cap U}(z)|^{-(1+\nu)}.$$
Now fix an $M>0$ which satisfies
$$\int_{M}^{\infty} \left(\frac{2 \beta}{r-(2 \alpha+\kappa)}\right)^{1+\nu}<\frac{\epsilon}{8c_1c_2}.$$
%and since convex domain is uniformly squeezing we have $$\frac{1}{c}k_{\Omega\cap U}(z,v)\leq b_{\Omega\cap U}(z,v)\leq k_{\Omega\cap U}(z,v)$$ where $b_{\Omega\cap U}(z,v)$ is the Bergman metric.
%Otherwise ,$$\frac{1}{c'}b_{\Omega\cap U}(z,v)\leq b_{\Omega}(z,v)\leq b_{\Omega\cap U}(z,v)$$
Then, since $\sigma$ is a geodesic,
we deduce that
\begin{align*}
1\geq k_{\Omega}(\sigma(t);\sigma'(t))&\geq \frac{1}{c_1}k_{\Omega\cap U}(\sigma(t);\sigma'(t))\geq\frac{1}{2c_1}\frac{|\sigma'(t)|}{\delta_{\Omega\cap U}(\sigma(t);\sigma'(t))}\\
&\geq \frac{1}{2c_1c_2}|\sigma'(t)||\log\delta_{\Omega\cap U}(\sigma(t)|^{1+\nu}\\
&= \frac{1}{2c_1c_2}|\sigma'(t)||\log\delta_{\Omega}(\sigma(t)|^{1+\nu}.
\end{align*}
So, we obtain that
$$\epsilon=\left\|\sigma(0)-\sigma\left(T_{0}\right)\right\| \leq \int_{0}^{T_{0}}\left\|\sigma^{\prime}(t)\right\| d t \leq 2c_1c_2 \int_{0}^{T_{0}} \frac{1}{|\log\delta_{\Omega}(\sigma(t))|^{1+\nu}} dt.$$
Then
\begin{align*}
\epsilon&\leq 2c_1c_2 \int_{\left[0, T_{0}\right] \cap(\tau-M, \tau+M)} \frac{1}{|\log\delta_{\Omega}(\sigma(t))|^{1+\nu}} d t\\
&+2c_1c_2 \int_{\left[0, T_{0}\right] \cap(\tau-M, \tau+M)^{c}} \frac{1}{|\log\delta_{\Omega}(\sigma(t))|^{1+\nu}} d t\\
&\leq 4 c_1c_2 M \frac{1}{(\log\delta_{\Omega}(\sigma(\tau)))^{1+\nu}}+4 c_1c_2 \int_{M}^{\infty} \left(\frac{ 2\beta}{r-(2\alpha+\kappa)}\right)^{1+\nu} d r\\
&\leq 4 c_1c_2 M \frac{1}{(\log\delta_{\Omega}(\sigma(\tau)))^2}+\frac{\epsilon}{2}.
\end{align*}
Therefore
$$\delta_{\Omega}(\sigma(\tau))\geq \exp\left(\left(\frac{\epsilon}{8c_1c_2M}\right)^{\frac{1}{1+\nu}}\right).$$
Although $\tau\in[0, T_0]$ depends on the specific $(1, \kappa)$-almost-geodesic $\sigma$, the lower bound for $\delta_{\Omega}(z)$ is independent of $z, ω $ and the $(1, \kappa)$-almost-geodesic $\sigma$ joining these two points.
This completes the proof.
\end{proof}
This allows us to make a uniform choice of A $\subseteq \Omega\cap U$ as desired.
%\begin{lemma}\label{geodesic}
%$\Omega$ is a bounded log-type convex domain, $\sigma_n|_{[a_n,b_n]}$ is a sequence of $K$-almost geodesic, if
%$$\lim|\sigma_n(a_n)-\sigma_n(b_n)|>0$$
%then there exists $T_n\in[a_n,b_n]$ such that $\sigma_n(t+T_n)$ converges locally uniformly to a $K$-almost geodesic $\sigma:\mathbb{R}\rightarrow\Omega$ up to a subsequence.
%\end{lemma}
%\begin{proof}
%Suppose that $\sigma_n(a_n)\rightarrow \xi\in\partial\Omega$ and $\sigma_n(b_n)\rightarrow\eta\in\partial\Omega$ and $\xi\neq\eta$, by the Lemma $\ref{visibility}$, there exists compact set $K\subset\Omega$, such that $\sigma_n\cap K\neq\emptyset$, so $\exists\alpha_n\in[a_n,b_n]$ such that $\sigma_n(\alpha_n)\in K$.
%
%Define
%$$\widetilde{\sigma}_n(t)=\sigma_n(t+\alpha_n)$$
%Since $\tilde{\sigma}_n(t)$ is $K$-Lipschitz, $\widetilde{\sigma}_n(t)$ converges to a $K$-almost geodesic $\sigma(t)$ locally uniformly up to a subsequence.
%
%\end{proof}
\begin{lemma}\label{localdistance1}
Suppose that $\Omega$ is a bounded domain with $\Omega\cap U$ log-type convex. For any $\xi\in\partial\Omega\cap U$, there exists $\epsilon>0 $ and $K'>0$ such that
$$K_{\Omega}(p,q)\leq K_{\Omega\cap U}(p,q)\leq K_{\Omega}(p,q)+K',$$
for any $p,q\in \Omega \cap B_{\epsilon}(\xi) \subset\Omega\cap U$.
\end{lemma}
%\begin{lemma}\label{est2}
%$\Omega$ is a bounded domain in $\mathbb{C}^n$, if $\Omega\cap U$ is log-type convex. For any $\epsilon>0$ there exists $\delta>0$ such that if $\xi\in\partial \Omega\cap U,p,q\in B_{\delta}(\xi)$, and $\sigma:[a,b]\rightarrow\Omega$ is a rough geodesic with $\sigma(a)=p$ and $\sigma(b)=q$, then $\sigma([a,b])\subset B_{\epsilon}(\xi)$
%\end{lemma}
To prove Lemma \ref{localdistance1}, the following lemma is needed. Note that $\Omega\cap U$ is log-type convex and $\xi\in\partial\Omega\cap U$.

\begin{lemma}\label{almostgeodesic}
For any $\epsilon>0$ with $B_{\epsilon}(\xi)\subset U$, there exists $\delta>0$ with the property:
For any $p,q\in B_{\delta}(\xi)$ and a $(1,\kappa)$-almost geodesic $\sigma$ which satisfies $\sigma(a)=p$ and $\sigma(b)=q$, there exists a rough geodesic $\widetilde{\sigma}$ joining $p$ and $q$ such that $\widetilde{\sigma}([a,b])\subset B_{\epsilon}(\xi)$.
%Suppose for a contradiction that the lemma does not hold for some $\epsilon>0$. Then for each $n$ there exists a point $\xi_{n}\in\partial\Omega$ and a $K$-almost geodesic $\sigma_{n}:[a_n,b_n]\rightarrow\Omega$ with $\sigma_{n}(a_n),\sigma_n(b_n)\in B_{1/n}(\xi_{n})$ and $\sigma_n(0)\in\Omega\backslash B_{\epsilon/2}(\xi)$.
\end{lemma}

\begin{proof}
Taking $\displaystyle{\delta=\frac{\epsilon}{4}}$,
we claim that either\\
(1) $\sigma|_{[a,b]}\subset B_{\epsilon}(\xi)$, or\\
(2) there exists $\alpha>0$ such that $\delta_{\Omega}(\sigma(a'))>\alpha$ and $\delta_{\Omega}(\sigma(b'))>\alpha$, where
$$a'=\inf\{t\in[a,b]:\sigma(t)\in\Omega\backslash B_{\epsilon/2}(\xi)\},$$
and
$$b'=\sup\{t\in[a,b]:\sigma(t)\in\Omega\backslash B_{\epsilon/2}(\xi)\}.$$

\bigskip
If (1) does not hold, we prove (2) holds.

We assume, by contradiction, that there exist sequences $\{a_m\}$ and $\{b_m\}$ and $(1,\kappa)$-almost geodesic $\sigma_m$ joining $a_m$ and $b_m$.
By passing to a subsequence we may assume that $\sigma(a_m)\in B_{\frac{\epsilon}{4}}(\xi)$ ,
$\sigma(b_m)\in B_{\frac{\epsilon}{4}}(\xi)$ and
$$
\sigma_{m}(a_m')\rightarrow\xi_1\in\partial \Omega\cap \partial B_{\frac{\epsilon}{2}}(\xi) \:\: \mbox{and} \:\: \sigma_{m}(b_m')\rightarrow\xi_2\in\partial \Omega\cap \partial B_{\frac{\epsilon}{2}}(\xi).
$$
By using Lemma $\ref{visibility}$, there exists a compact set $A\subset\Omega$ such that $\sigma_m|_{[a_m,a_m']}\cap A\neq\emptyset$ and $\sigma_m|_{[b_m',b_m]}\cap A\neq\emptyset$. Suppose $\sigma_m(t_m^1)\in\sigma_m|_{[a_m,a_m']} $ and $\sigma_m(t_m^2)\in\sigma_m|_{[b_m',b_m]}$ which satisfy $\sigma_m(t_m^1)\in A$ and $\sigma_m(t_m^2)\in A$.

On one hand, we have
\begin{align*}
K_{\Omega}&(\sigma_{m}(a_m),\sigma_m (b_m))\\
&=K_{\Omega}(\sigma_{m}(a_m),\sigma_m (t_m^1))+K_{\Omega}(\sigma_{m}(t_m^1),\sigma_m (a_m'))
+K_{\Omega}(\sigma_{m}(b_m'),\sigma_m (t_m^2))\\
&+K_{\Omega}(\sigma_{m}(t_m^2),\sigma_m (b_m))\\
&\geq K_{\Omega}(\sigma_{m}(a_m),\sigma_m (t_m^1))+K_{\Omega}(\sigma_{m}(t_m^2),\sigma_m (b_m))+\frac{1}{2}\log\frac{1}{\delta_{\Omega}(\sigma_{m}(a_m'))}\\
&+
\frac{1}{2}\log\frac{1}{\delta_{\Omega}(\sigma_{m}(b_m'))}-C.
\end{align*}
On the other hand, joining $\sigma_m(t_m^1)$ and $\sigma_n(t_m^2)$ by a line segment $l$, we define
$$\widetilde{\sigma_m}(t)=
\begin{cases}
\sigma(t) \ \ \ \ \ t\in\mathbb{R}\backslash(t_m^1,t_m^2)\\
l(t) \ \ \ \ \ t\in[t_m^1,t_m^2].
\end{cases}
$$
Then we deduce that
\begin{align*}
L(\sigma_m|_{[a_m,b_m]})&-L(\widetilde{\sigma_m}|_{[a_m, b_m]})\\
&\geq\frac{1}{2}\log\frac{1}{\delta_{\Omega}(\sigma_{m}(a_m'))}+\frac{1}{2}\log\frac{1}{\delta_{\Omega}(\sigma_{m}(b_m'))}+C'
\rightarrow\infty,
\end{align*}
which contradicts the fact the $\sigma_m$ is a $(1,\kappa)$-almost geodesic.
Thus, it implies that there exists $\alpha>0$ such that $\delta_{\Omega}(\sigma(a'))\geq \alpha$ and $\delta_{\Omega}(\sigma(b'))\geq \alpha$, which means $\sigma(a')$ and $\sigma(b')$ are contained in a compact set in $\Omega\cap U$.

Therefore, by joining $\sigma_{m}(a_m')$ and $\sigma_{m}(b_m')$ by a line segment and re-parameterizing the new path, we get a $C''$-rough geodesic, which proves the lemma.
\end{proof}

%Thus there exists $C''>0$ such that $K_{\Omega\cap U}(\sigma(a'),\sigma(b'))\leq C''$.
%$len(l)$ in the $K_{\Omega\cap U}$ distance is less than $C$ where $l$ is the line segment joining $\sigma(a')$ and $\sigma(b')$.
\textit{Proof of Lemma \ref{localdistance1}.}

For any $p,q\in B_{\epsilon}(\xi)$ and an almost geodesic $\sigma$ with $\sigma(a)=p, \: \sigma(b)=q$,
%from the construction we know that it is different with the geodesic $\sigma$ only on a compact set $[a',b']$. Moreover since $\sigma(a')$ and $\sigma(b')$ are both contained in a compact set in $\Omega\cap U$, there exists $C>0$ such that the length of $\sigma|_{[a',b']}$ in the $K_{\Omega\cap U}$ metric is less than $C$.
we then obtain
\begin{align*}
K_{\Omega\cap U}&(p,q)\\
&\leq K_{\Omega\cap U}(\sigma(a),\sigma(a'))+K_{\Omega\cap U}(\sigma(a'),\sigma(b'))+K_{\Omega\cap U}(\sigma(b'),\sigma(b))\\
&\leq K_{\Omega}(\sigma(a),\sigma(a'))+K_{\Omega}(\sigma(b'),\sigma(b))+2K+C''\\
&\leq K_{\Omega}(p,q)+K',
\end{align*}
where $K'=2K+C''$ and $K$ is the constant defined in Lemma $\ref{local}$.

The proof is complete. \hfill$\square$

\bigskip
\textit{Proof of Theorem \ref{localdistance}.}

If $\overline V \subset \Omega$, there is nothing to do.

Otherwise, for any $\xi\in U\cap \partial \Omega$ there is a $\delta>0$ such that Theorem \ref{localdistance1} holds in $B_{\delta}(\xi)$. Moreover for any $\xi'\in B_{\frac{\delta}{2}}(\xi)\cap \partial\Omega$, Theorem \ref{localdistance1} also holds in $B_{\frac{\delta}{2}}(\xi')\subset B_{\delta}(\xi)$. Noting that $\partial \Omega \cap \overline V$ is compact,
 we complete the proof. \hfill$\square$

\bigskip
Next we will give a similar localization of a log-type $\mathbb{C}$-convex domain with Dini-smooth boundary. Recall that a domain $\Omega$ is called $\mathbb{C}$-convex if the non-empty intersection with a complex line is simply connected. A log-type $\mathbb{C}$-convex domain is a $\mathbb{C}$-convex domain which also satisfies (1).

At first we need some estimates of the Kobayashi metrics in $\mathbb{C}$-convex domains studied in \cite{Nikolov2015Estimates,Nikolov2015The,Nikolov2008Estimates}.
\begin{lemma}[\cite{Nikolov2008Estimates},\cite{Nikolov2015The}]
If $\Omega$ is an $\mathbb{C}$-convex bounded domain, and $p,q\in\Omega$ are distinct and $v\in\mathbb{C}^n$, then
\begin{align}
k_{\Omega}(p;v)\geq\frac{1}{4}\frac{|v|}{\delta_{\Omega}(p;v)},
\end{align}
$$K_{\Omega}(p,q)\geq\frac{1}{4}\left|\log\left(1+\frac{|p-q|}{\min\{\delta_{\Omega}(p;p-q),\delta_{\Omega}(q;p-q)\}}\right)\right|.$$
\end{lemma}
\begin{lemma}[\cite{Nikolov2015Estimates}, Theorem 7]\label{est}
If $\xi$ be a Dini-smooth boundary point of a domain $\Omega$ in $\mathbb{C}^n$, then there exists $c>0$ and a neighbourhood $U$ of $\xi$ such that
$$K_{\Omega}(p,q)\leq\log\left(1+\frac{2|p-q|}{\sqrt{\delta_{\Omega}(p)\delta_{\Omega}(q)}}\right), \:\:\: \forall p, \: q\in \Omega\cap U.$$
\end{lemma}
Therefore, if $\Omega\cap U$ is a $\mathbb{C}$-convex domain with Dini-smooth boundary, by fixing some $z_0\in\Omega\cap U$ and letting $L$ be the complex line containing $z_0$ and $z$ and $\xi\in L\cap \partial(\Omega\cap U)$, then there exists a $C>0$ such that
\begin{align}
\frac{1}{4}\left|\log\left(\frac{|z_0-\xi|}{|z-\xi|}\right)\right|\leq K_{\Omega}(z,z_0)\leq\frac{1}{2}\log\frac{1}{\delta_{\Omega}(z)}+C.
\end{align}
Moreover, if $\Omega\cap U$ is a log-type $\mathbb{C}$-convex domain with Dini-smooth boundary, then $\Omega$ also has the locally visibility property.

Similarly we can repeat the proof of Theorem \ref{localdistance} and obtain the following result.
\begin{thm}\label{c-convex}
If $\Omega$ is a bounded domain and $\Omega\cap U$ is log-type $\mathbb{C}$-convex with Dini-smooth boundary, then for any $V\subset U$ with $\overline V \subset U$, there exists $K>0$ such that, for every $p,q\in V\cap \Omega$,
$$K_{\Omega}(p,q)\leq K_{\Omega\cap U}(p,q)\leq K_{\Omega}(p,q)+K.$$
\end{thm}
Since the proof of Theorem \ref{c-convex} does not add any further insights beyond those
seen in the proof of Lemma \ref{localdistance1}, we shall omit its proof.
Noting that (2, 3, 4) can be replaced by (12, 13), the inequalities in (5) and (9)-(11) still hold. Therefore, the arguments can go through without any difficultes.
\section{Application}
In this section, we prove the local version of continuous extension and some results about the Teichm\"{u}ller space. At first we need a lemma of Gromov product.
\begin{lemma}\label{product}
Suppose that $\Omega$ is a convex domain with Dini-smooth boundary. Fixing $z_0 \in \Omega$, if $x_n\rightarrow\xi$ and $y_n\rightarrow\xi$, then we have
$$(x_n|y_n)_{z_0}\rightarrow\infty.$$
\end{lemma}
\begin{proof}
In \cite[Theorem 4.1]{Zimmer2017The}, the domain $\Omega$ is assumed to be a $C^{1,\epsilon}$ convex domain. By using Lemma $\ref{est}$, one can repeat the proof with the weaker assumption.
\end{proof}
Next we give a result on the local version of continuous extensions
of roughly isometric embedding, which is a re-statement of Corollary
\ref{locallocal}.
\begin{cor}
Let $\Omega_1, \:\Omega_2\subset\mathbb{C}^n$ be bounded domains and $f:\Omega_1\rightarrow\Omega_2$ be a rough isometry with respect to the Kobayashi distance.
Suppose $\Omega_1\cap U_1$ is a log-type convex domain with Dini-smooth boundary and $\Omega_2\cap U_2$ is a log-type $\mathbb{C}$-convex domain with Dini-smooth boundary.

If $\xi\in\partial\Omega_1\cap U_1$, then $\zeta\in\partial\Omega_2\cap U_2$ and $\zeta\in\mathcal{C}(f,\xi)$,
then $f$ extends continuously to $\xi$.
\end{cor}
\begin{proof}
Note that, by definition, there exists $C > 0$ such that for any $x,y\in\Omega_1$ $$K_{\Omega_1}(x,y)-C\leq K_{\Omega_2}(f(x),f(y))\leq K_{\Omega_1}(x,y)+C.$$
%By Theorem $\ref{c-convex}$, we have
%$$K_{\Omega_i}(p,q)\leq K_{\Omega_i\cap U_i}(p,q)\leq K_{\Omega_i}(p,q)+K$$
%also by direct computation $$\left|(p|q)_o -(p|q)_{o'}\right|\leq K_{\Omega}(p,q)$$ we have
Therefore, by fixing $z_0\in\Omega_1\cap U_1$,
$$\left|(f(x)|f(y))_{f(z_0)}^{\Omega_2}-(x|y)_{z_0}^{\Omega_1}\right|\leq\frac{3}{2}C.$$
If $x_n\rightarrow \xi$, $y_n\rightarrow \xi$ and $f(x_n)\rightarrow\zeta$, then we will show that $f(y_n)\rightarrow\zeta$.

Conversely, suppose that $f(y_n)\rightarrow\zeta'\in\partial\Omega_2$ and $\zeta\neq\zeta'$. From Lemma $\ref{visibility}$, it follows that there is a compact set $A\subset\Omega_2\cap U_2$ such that the $(1,\kappa)$-almost geodesic $\sigma_n(t)$ joining $f(x_n)$ and $f(y_n)$ intersects $A$. Fix $t_n\in $dom$(\sigma_n)$ such that $\sigma_n(t_n)\in A$.
Letting $R=\max\limits_{z\in A}K_{\Omega_2}(z,f(z_0))$, then we have
\begin{align*}
\Big(&f(x_n)|f(y_n)\Big)_{f(z_0)}^{\Omega_2}\\
&= \frac{1}{2}\left(K_{\Omega_2}(f(x_n),f(z_0))+K_{\Omega_2}(f(y_n),f(z_0))-K_{\Omega_2}(f(x_n),f(y_n))\right)\\
&\leq\frac{1}{2}\big(K_{\Omega_2}(f(x_n),f(z_0))+K_{\Omega_2}(f(y_n),f(z_0))-K_{\Omega_2}(f(x_n),f(\sigma_n(t_n)))\\
&-K_{\Omega_2}(f(\sigma_n(t_n)),f(y_n))\big)+\kappa\\
&\leq K_{\Omega_2}(f(z_0),f(\sigma_n(t_n)))+\kappa+R\\
&<\infty.
\end{align*}
On the other hand, in view of Theorem \ref{localdistance},
$$(x_n|y_n)_{z_0}^{\Omega_1\cap U_1}-K\leq(x_n|y_n)_{z_0}^{\Omega_1}\leq(x_n|y_n)_{z_0}^{\Omega_1\cap U_1}+K,$$
Then by Lemma \ref{product}, we have
$$(x_n|y_n)_{z_0}^{\Omega_1\cap U_1}\rightarrow\infty,$$
Thus
$$(x_n|y_n)_{z_0}^{\Omega_1}\rightarrow\infty,$$
which is a contradiction.

Therefore we have $\lim\limits_{z\rightarrow\xi}f(z)=\zeta$, which completes the proof.
\end{proof}
\begin{rmk}
 Note that in the above proof we only need $\Omega_2\cap U_2$ to be 'visible'. By using Lemma \ref{visibility}, the condition 'Dini-smooth' of $\Omega_2$ can be removed if $\omega_2\cap U_2$ is log-type convex instead of log-type $\mathbb{C}$-convex.
\end{rmk}

%\begin{cor}
%Suppose that $\Omega_1, \:\Omega_2$ are bounded domains in $\mathbb{C}^n$ and $f: \Omega_1\rightarrow\Omega_2$ is a rough isometry with respect to the Kobayashi distance, Furthermore, if $\Omega_1$ is convex with Dini-smooth boundary and $\Omega_2\cap U$ is log-type convex, and if $\xi\in\partial\Omega_1$, $\zeta\in\partial\Omega_2\cap U$ and $\zeta\in\mathcal{C}(f,\xi)$, then $f$ extends continously to $\xi$.
%\end{cor}
Next we prove a result on domains biholomorphic to the
$Teichm\ddot{u}ller$ space $T_{g, n}$. At first we need some lemmas.
\begin{defn}
For any domain $\Omega \subset \mathbb C^n$, a boundary point $\xi\in\partial\Omega$ is Alexandroff smooth if

(i) $\Omega$ is locally convex at $\xi$;

(ii) there exists $r>0$ such that $\Omega \cap B(\xi, r)$ is convex and $\partial \Omega \cap B(\xi, r)$ is the graph of a convex function $\psi: U \cap V \rightarrow \mathbb{R}_{+}$ which has a second order Taylor expansion at $\xi$. That is, if we assume without loss of generality that $\xi=0$ and $V=\left\{x_{n}=0\right\}$ is a supporting hyperplane for $\Omega \cap B(\xi, r)$, then we have
$$\psi\left(x_{1}, x_{2}, \ldots x_{n}\right)=\frac{1}{2} \sum_{i, j} H_{i, j} x_{i} x_{j}+o\left(\|x\|^{2}\right)$$
for some $n \times n$ symmetric matrix $H$ (which, for a genuine $C^{2}$ -function, is the
Hessian).
\end{defn}
\begin{thm}[Alexandroff\cite{alexandroff1939almost}]\label{second}
If $\Omega$ is a convex domain, then almost every boundary point is Alexandroff smooth in the above sense.
\end{thm}
By using Theorem $\ref{second}$, one can prove the following lemma directly.
\begin{lemma}\label{sphere}
Let $\Omega \subset \mathbb{R}^{n}$ be a domain and let $\xi \in \partial \Omega$ be an Alexandroff smooth point. Then $\xi$ has an interior sphere contact. Namely there is
a round sphere $S$ contained in $\overline{\Omega}$ such that $S \cap \overline{\Omega}=\{\xi\}.$
\end{lemma}

Note that the Teichm\"uller modular group $Mod_{g, n}$ isometrically acts on $T_{g, n}$ and this action is properly discontinuous. Furthermore, we have
\begin{prop}[\cite{knudsen1977projectivity}]
The $Teichm\ddot{u}ller$ space admits a finite volume quotient in the Kobayashi metric, i.e. the quotient space $T_{g, n}/Mod_{g, n}$ has finite Kobayashi volume.
\end{prop}

\begin{lemma}[Theorem 2,\cite{Kerckhoff1983The}]\label{finite}
Every finite subgroup of $Mod_{g, n}$ acting on $T_{g, n}$ has a fixed point.
\end{lemma}

Then we can follow the methods of Zimmer \cite{Zimmer2018Smoothly} and Gupta-Seshadri \cite{Gupta2017On} to prove the following result, which is a re-statement of Theorem
\ref{Teichmuller}.
\begin{thm}\label{teichmuller}
$Teichm\ddot{u}ller$ space $T_{g, n}$ can not be biholomorphic to a bounded domain in $\mathbb C^{3g-3+n}$ which is locally log-type convex at some boundary point.
\end{thm}
\begin{proof}
Assume, by contradiction, that the $Teichm\ddot{u}ller$ space $T_{g, n}$ is biholomorphic to $\Omega\subset \mathbb C^{3g-3+n}$ such that
$\Omega \cap U$ is log-type convex at some boundary point neighborhood $U$.
It is well know that $Mod_{g, n}$ has a finite index subgroup $N$ that acts freely on $T_{g, n}$. Let $\Gamma \subset Aut(\Omega)$ be the subgroup corresponding to $N \subset Mod_{g, n}$. Then the quotient space $\Omega/\Gamma$ is a manifold with finite $K\ddot{a}hler-Einstein$ volume.

From Lemma $\ref{sphere}$, it follows that we can choose a boundary point in $U\cap \partial \Omega$, denoted by $\xi$, which admits an interior sphere contact.
Then we will prove the following

({\bf Claim}) There exist $\left\{\gamma_{n}\right\} \subseteq \Gamma$ such that $\gamma_{n}\left(z_{0}\right) \rightarrow \xi$ for some fixed point $z_0 \in \Omega$.

\bigskip
Therefore, by using the scaling method of Kim-Krantz-Pinchuk
(refer to section 9.2.5 of \cite{greene2011geometry}), and  in a similar way as used by Gupta-Seshadri in
Section 3.4 of \cite{Gupta2017On}, we obtain that $Mod_{g, n}$ has a one parameter subgroup which contradicts to the discreteness of
$Mod_{g, n}$.

It completes the proof of Theorem \ref{teichmuller}.
\end{proof}
\bigskip

\textit{Proof of the Claim.}

For any $\Omega \supset \{y_{n}\} \rightarrow \xi$, define
$$\delta_{n}=\min _{\gamma \in \Gamma \backslash\{i d\}} K_{\Omega}\left(y_{n}, \gamma y_{n}\right).$$
Then the quotient map $\pi : \Omega \rightarrow \Omega / \Gamma$
restricts to an embedding on the ball $B_{n}=\left\{z \in \Omega : K_{\Omega}\left(z, y_{n}\right)<\frac{\delta_{n}}{2}\right\} .$ The $Teichm\ddot{u}ller$ space admits the uniform squeezing property, so does $\Omega$. From Lemma $\ref{squeezing}$, it follows that
there exists $C$ and $\epsilon$ such that
$$\operatorname{vol}\left(\pi\left(B_{n}\right)\right) \geq C \min \left\{\epsilon^{2 n}, \delta_{n}^{2 n}\right\}.$$
After passing to a subsequence we can assume that
$$\lim _{n \rightarrow \infty} \delta_{n}=\delta \geq 0.$$
Now we have the following three cases $(a), (b), (c)$:

\bigskip
(a) If $\delta>0$, then the set $\left\{\pi\left(y_{n}\right) : n \in \mathbb{N}\right\}$ must be relatively compact in $\Omega / \Gamma$. Otherwise, $\bigcup\limits_{n\in\mathbb{N}} B_{n}$
would admit infinite volume. Therefore, for each $n$, we can get $\gamma_n \in \Gamma$ such that $\gamma_ny_n\rightarrow z_0 \in \Omega$, which implies that $\gamma_{n}^{-1} z_0 \rightarrow \xi$, as desired.

\bigskip
(b) If $\delta=0$ and the set $\{\gamma_{n}\}_{n\geq 1}$ is infinite, noting that $\Gamma$ is discrete, we can assume that
$\gamma_{n} z_{0} \rightarrow \xi' \in \partial \Omega$ up to a subsequence for some fixed $z_{0}$. We claim that $\xi'=\xi$.
Suppose that $\xi\neq \xi' \in \partial \Omega$. By lemma $\ref{visibility}$, it follows that there exists compact set $A$ such that $\sigma_n\cap A\neq\emptyset$, where $\sigma_n$ joining $y_n$ and $\gamma_nz_0$ with $\sigma_n(t_n)\in A$.
Noting that $K_{\Omega}\left(z_{0}, \gamma_{n} z_{0}\right) \rightarrow \infty$, we thus have
\begin{align*}
K_{\Omega}&\left(\gamma_{n} z_{0}, z_{0}\right)+K_{\Omega}\left(z_{0}, y_{n}\right)-K_{\Omega}\left(\gamma_{n} z_{0}, y_{n}\right)\\
&\leq K_{\Omega}\left(\gamma_{n} z_{0}, z_{0}\right)+K_{\Omega}\left(z_{0}, y_{n}\right)-K_{\Omega}\left(\gamma_{n} z_{0},\sigma_n(t_n)\right)-K_{\Omega}\left(\sigma_n(t_n),y_n\right)\\
&\leq 2K_{\Omega}\left(\sigma_n(t_n),z_0\right)\leq R,
\end{align*}
where $R=2\max\{K_{\Omega}(z_0,z):z\in A\}$.
However,
\begin{align*}
K_{\Omega}\left(\gamma_{n} z_{0}, z_{0}\right) &+K_{\Omega}\left(z_{0}, y_{n}\right)-K_{\Omega}\left(\gamma_{n} z_{0}, y_{n}\right) \\ &=K_{\Omega}\left(\gamma_{n} z_{0}, z_{0}\right)+K_{\Omega}\left(z_{0}, y_{n}\right)-K_{\Omega}\left(z_{0}, \gamma_{n}^{-1} y_{n}\right) \\ & \geq K_{\Omega}\left(\gamma_{n} z_{0}, z_{0}\right)-K_{\Omega}\left(y_{n}, \gamma_{n}^{-1} y_{n}\right) \rightarrow \infty,
\end{align*}
from which we deduce that $\xi'=\xi$.

\bigskip
(c) If the set $\{\gamma_{n}\}_{n\geq 1}$ is finite, without loss of generality we assume $\gamma_{n}=\gamma$ for all $n \in \mathbb{N}$.
If $\gamma^{n}$ is relatively compact, noting that $\Gamma$ is discrete, then $\gamma$ is of finite order. Lemma $\ref{finite}$ immediately
implies that $\gamma$ has a fixed point in $\Omega$, which is impossible.

Therefore, $\gamma^{n} z_{0}\rightarrow \xi'\in \partial \Omega$. We assume that $\xi'\neq\xi \in \partial \Omega$. Using the fixed point $z_{0}$, we consider the function
$$b_{n}(z)=K_{\Omega}\left(z, y_{n}\right)-K_{\Omega}\left(y_{n}, z_{0}\right).$$
Since $b_{n}\left(z_{0}\right)=0$ and each $b_n(z)$ is $1$-Lipschitz, we can pass to a subsequence such
that $b_{n} \rightarrow b$ locally uniformly. Therefore
$$b(\gamma z)=\lim _{n \rightarrow \infty} K_{\Omega}\left(\gamma z, y_{n}\right)-K_{\Omega}\left(y_{n}, z_{0}\right)=\lim _{n \rightarrow \infty} K_{\Omega}\left(z, \gamma^{-1} y_{n}\right)-K_{\Omega}\left(y_{n}, z_{0}\right)=b(z),$$
which implies that $b\left(\gamma^{n} z_{0}\right)=b\left(z_{0}\right)=0$ for all $n \in \mathbb{N}$.

On the other hand, there exists a $R>0,$

$$K_{\Omega}\left(\gamma^{n} z_{0}, z_{0}\right)+K_{\Omega}\left(z_{0}, y_{n}\right)-K_{\Omega}\left(\gamma^{n} z_{0}, y_{n}\right) \leq R.$$
Then
$$b\left(\gamma^{n} z_{0}\right)=K_{\Omega}\left(\gamma^{n} z_{0}, y_{n}\right)-K_{\Omega}\left(z_{0}, y_{n}\right) \geq K_{\Omega}\left(\gamma^{n} z_{0}, z_{0}\right)-R \rightarrow \infty,$$
which is a contradiction. It completes the proof that $\gamma^nz_0\rightarrow\xi$ and hence proves (c).

Therefore, the proof of the claim is complete. \hfill$\square$

\vspace{0.3cm} \noindent{\bf Acknowledgement}. The authors
would like to thank Yunhui Wu and Liyou Zhang for their precious advices and for many stimulating discussions. We would also like to thank the referee for a careful reading and valuable comments.

\bibliography{Localization}
\bibliographystyle{plain}{}
\end{document}